\documentclass[12pt]{amsart}
\usepackage{amsmath,amsthm,amsfonts,amssymb,mathrsfs}
\date{\today}

\usepackage[cp1251]{inputenc}         

\usepackage[ukrainian]{babel} 
\usepackage{amsmath,amsthm,amsfonts,amssymb,mathrsfs}
\usepackage{eucal}

\input xymatrix
\xyoption{all}

\usepackage{color}

\usepackage{amssymb,cite}
\usepackage[colorlinks,plainpages,citecolor=magenta, linkcolor=blue, backref]{hyperref}

\usepackage{hyperref}

\usepackage{tikz}

\newtheorem{theorem}{Теорема}

\newtheorem{proposition}{Твердження}
\newtheorem{corollary}{Наслiдок}
\newtheorem{lemma}{Лема}
\theoremstyle{definition}
\newtheorem{example}{Приклад}
\newtheorem{remark}{Зауваження}

\newtheorem{definition}[theorem]{Означення}

\begin{document}

\title[Про компактнi топологi\"{\i} на напiвгрупi скiнченних часткових порядкових iзоморфiзмiв~]{Про компактнi топологi\"{\i} на напiвгрупi скiнченних часткових порядкових iзоморфiзмiв обмеженого рангу нескiнченно\"{\i} лiнiйно впорядковано\"{\i}  множини}

\author[Олег~Гутік, Максим Щипель]{Олег~Гутік, Максим Щипель}
\address{Львівський національний університет ім. Івана Франка, Університецька 1, Львів, 79000, Україна}
\email{oleg.gutik@lnu.edu.ua, maksym.shchypel@lnu.edu.ua}



\begin{abstract}
Ми  досліджуємо топологізацію напівгрупи $\mathscr{O\!\!I}\!_n(L)$ скiнченних часткових порядкових iзоморфiзмiв обмеженого рангу нескiнченно\"{\i} лiнiйно впорядковано\"{\i}  множини $(L,\leqslant)$. Зокрема, доведено, що кожна $T_1$ ліво-топологічна (право-топологічна) напівгрупа $\mathscr{O\!\!I}\!_n(L)$ є цілком гаусдорфовим, урисонівським, цілком відокремлюваним, розрідженим простором. Доведено, що на напівгрупі $\mathscr{O\!\!I}\!_n(L)$ існує єдина гаусдорфова зліченно компактна (псевдокомпактна) трансляційно-неперервна топологія, яка є компактною, а також, що компактифікація Бора гаусдорфової топологічної напівгрупи $\mathscr{O\!\!I}\!_n(L)$ є тривіальною напівгрупою.

\bigskip
\noindent
\emph{Oleg Gutik, Maksym Shchypel, \textbf{On compact topologies on the semigroup of finite partial order isomorphisms of a bounded rank of an infinite linear ordered set}.}

\smallskip
\noindent
We study topologization of the semigroup $\mathscr{O\!\!I}\!_n(L)$ of finite partial order isomorphisms of a bounded rank of an infinite linear ordered set $(L,\leqslant)$. In particular we show that every $T_1$ left-topological (right-topological) semigroup $\mathscr{O\!\!I}\!_n(L)$ is a  completely Hausdorff, Urysohn, totally separated, scattered space. We prove that on the semigroup $\mathscr{O\!\!I}\!_n(L)$ admits a unique Hausdorff countably compact (pseudocompact) shift-continuous topology  which is compact, and the Bohr compactification of a Hausdorff topological semigroup $\mathscr{O\!\!I}\!_n(L)$ is the trivial semigroup.

\end{abstract}

\maketitle


\section{Вступ}\label{section-1}

У цій праці ми користуємося термінологією з монографій \cite{Carruth-Hildebrant-Koch=1983, Clifford-Preston-1961, Clifford-Preston-1967, Engelking=1989, Lawson-1998, Petrich-1984, Ruppert=1984}.

Якщо визначене часткове відображення $\alpha\colon X\rightharpoonup Y$ з множини $X$ у множину $Y$, то через $\operatorname{dom}\alpha$ i $\operatorname{ran}\alpha$ будемо позначати його \emph{область визначення} та \emph{область значень}, відповідно, а через $(x)\alpha$ і $(A)\alpha$ --- образи елемента $x\in\operatorname{dom}\alpha$ та підмножини $A\subseteq\operatorname{dom}\alpha$ при частковому відображенні $\alpha$, відповідно.
Також через $\operatorname{rank}\alpha$ будемо позначати ранг часткового відображення $\alpha\colon X\rightharpoonup Y$, тобто $\operatorname{rank}\alpha=|\operatorname{dom}\alpha|$.

Нехай $(X,\leqq)$~--- частково впорядкована множина. Для довільного елемента $x$ з $(X,\leqq)$ позначимо
\begin{equation*}
  {\uparrow}_{\leqq}x=\left\{y\in X\colon x\leqq y\right\}, \qquad {\downarrow}_{\leqq}x=\left\{y\in X\colon y\leqq x\right\}  \qquad \hbox{i} \qquad {\downarrow}_{\leqq}^\circ x={\downarrow}_{\leqq}x\setminus\{x\}.
\end{equation*}

Надалі через $E(S)$ позначатимемо множину ідемпотентів напівгрупи $S$. Напівгрупа ідемпотентів називається \emph{в'язкою}, а комутативна напівгрупа ідемпотентів --- \emph{напівґраткою}.

Якщо $S$~--- напівгрупа, то на $E(S)$ визначено частковий порядок:
$
e\preccurlyeq f
$   тоді і лише тоді, коли
$ef=fe=e$.
Так означений частковий порядок на $E(S)$ називається \emph{при\-род\-ним}.

Напівгрупа $S$ називається \emph{інверсною}, якщо для довільного елемента $s\in S$ існує єдиний елемент $s^{-1}\in S$ такий, що $ss^{-1}s=s$ i $s^{-1}ss^{-1}=s^{-1}$~\cite{Wagner-1952}. В інверсній напівгрупі $S$ вище означений елемент $s^{-1}$ називається \emph{інверсним до} $s$.

Означимо відношення $\preccurlyeq$ на інверсній напівгрупі $S$ так:
$
    s\preccurlyeq t
$
тоді і лише тоді, коли $s=te$, для деякого ідемпотента $e\in S$. Так означений частковий порядок назива\-єть\-ся \emph{при\-род\-ним част\-ковим порядком} на інверсній напівгрупі $S$~\cite{Wagner-1952}. Очевидно, що звуження природного часткового порядку $\preccurlyeq$ на інверсній напівгрупі $S$ на її в'язку $E(S)$ є при\-род\-ним частковим порядком на $E(S)$.

Через $\mathscr{I}(X)$ позначимо множину всіх часткових взаємно однозначних перетворень множини $X$ разом з такою напівгруповою операцією
\begin{equation*}
    x(\alpha\beta)=(x\alpha)\beta, \quad \mbox{якщо} \quad
    x\in\operatorname{dom}(\alpha\beta)=\{
    y\in\operatorname{dom}\alpha\colon
    y\alpha\in\operatorname{dom}\beta\}, \qquad \mbox{для} \quad
    \alpha,\beta\in\mathscr{I}_\lambda.
\end{equation*}
Напівгрупа $\mathscr{I}(X)$ називається  \emph{симетричною інверсною напівгрупою} або \emph{симетричним інверсним моноїдом} над множиною $X$~(див.  \cite{Clifford-Preston-1961}). Симетрична інверсна напівгрупа введена В. В. Вагнером у працях~\cite{Wagner-1952a, Wagner-1952} і вона відіграє важливу роль у теорії напівгруп. Надалі, якщо для $\alpha,\beta\in \mathscr{I}(X)$ виконуються умови $\operatorname{dom}\alpha\subseteq \operatorname{dom}\beta$ i $(x)\beta=(x)\alpha$ для довільного $x\in \operatorname{dom}\alpha$, то будемо писати $\alpha\subseteq \beta$. Для довільного натурального числа $n$ позначимо $\mathscr{I}^n(X)=\{\alpha\in \mathscr{I}(X)\colon \operatorname{rank}\alpha\leq n\}$. Напівгрупа $\mathscr{I}^n(X)$ називається \emph{симетричною інверсною напівгрупою} або \emph{симетричним інверсним моноїдом рангу} $\leq n$ над множиною $X$ \cite{Gutik-Lawson-Repov=2009, Gutik-Reiter=2009, Gutik-Reiter=2010}. Очевидно, що $\mathscr{I}^n(X)$ --- інверсна піднапівгрупа в $\mathscr{I}(X)$, більше того $\mathscr{I}^n(X)$ --- ідеал в $\mathscr{I}(X)$.

Надалі будемо вважати, що $(L,\leqslant)$~--- нескінченна лінійно впорядкована множина. Для елементів $x,y\in L$ умову $x\leqslant y$ i $x\neq y$ записуватимемо так: $x<y$. Елемент $\alpha\in\mathscr{I}(L)$ називається \emph{частковим порядковим iзоморфiзмом}, якщо  для довільних $x_1,x_2\in \operatorname{dom}\alpha$ з $x_1\leqslant x_2$ випливає $(x_1)\alpha\leqslant (x_2)\alpha$. Позаяк $\leqslant$~--- лінійний порядок на $L$, то для $\alpha\in\mathscr{I}(L)$ і для до\-віль\-них $x_1,x_2\in \operatorname{dom}\alpha$ з $(x_1)\alpha\leqslant (x_2)\alpha$ випливає $x_1\leqslant x_2$. Очевидно, що композиція часткових порядкових iзоморфiзмiв лінійно впорядкованої множини $(L,\leqslant)$ є частковим порядковим iзоморфiзмом, і обернене часткове відображення до часткового порядкового iзоморфiзму є знову частковим порядковим iзоморфiзмом. Через $\mathscr{O\!\!I\!}(L)$ позначимо напівгрупу всіх часткових порядкових iзоморфiзмів лінійно впорядкованої множини $(L,\leqslant)$. Очевидно, що $\mathscr{O\!\!I\!}(L)$~--- інверсна напівгрупа симетричного інверсного моноїда $\mathscr{I}(L)$ над множиною $L$.

Для довільного натурального числа $n$ визначимо
\begin{equation*}
  \mathscr{O\!\!I\!}_n(L)=\left\{\alpha\in\mathscr{O\!\!I\!}(L)\colon |\operatorname{dom}\alpha|\leq n\right\}.
\end{equation*}
Очевидно, що $\mathscr{O\!\!I\!}_n(L)$~--- інверсна піднапівгрупа симетричного інверсного моноїда $\mathscr{O\!\!I\!}(L)$ над множиною $L$. Надалі напівгрупу $\mathscr{O\!\!I}\!_n(L)$ будемо називати \emph{напiвгрупою скiнченних часткових порядкових iзоморфiзмiв рангу $\leq n$ лiнiйно впорядкованої  множини} $(L,\leqslant)$ \cite{Gutik-Shchypel=2024}. Далі в тексті, якщо елемент $\alpha$ напівгрупи $\mathscr{O\!\!I}\!_n(L)$ відображає $x_1$ у $y_1$, $\ldots$, $x_k$ у $y_k$, де $x_1,\ldots x_k,y_1,\ldots y_k\in L$, $x_1<\cdots< x_k$, $y_1<\cdots< y_k$, $k\leq n$, то записуватимемо так:
\begin{equation*}
  \alpha=
\begin{pmatrix}
x_1 & \cdots & x_k\\
y_1 & \cdots & y_k
\end{pmatrix}.
\end{equation*}
Порожнє часткове перетворення множини $L$ будемо позначати символом $\boldsymbol{0}$. Очевидно, що $\boldsymbol{0}$ --- нуль напівгрупи $\mathscr{O\!\!I}\!_n(L)$. Очевидно, що $\mathscr{O\!\!I}\!_n(L)$ --- інверсна піднапівгрупа в $\mathscr{I}^n(L)$, а також, що напівгрупи $\mathscr{I}^1(L)$ і $\mathscr{O\!\!I}\!_1(L)$ ізоморфні напівгрупі $L\times L$-матричних одиниць.

У \cite{Gutik-Shchypel=2024} ми досліджували алгебричні властивості напівгрупи $\mathscr{O\!\!I\!}_n(L)$. Зокрема описано її ідемпотенти, природний частковий порядок та відношення Ґріна на $\mathscr{O\!\!I\!}_n(L)$. Доведено, що напівгрупа $\mathscr{O\!\!I\!}_n(L)$ стійка та містить щільний ряд ідеалів, а також, що всі конґруенції на напівгрупі $\mathscr{O\!\!I\!}_n(L)$ є конґруенціями Ріса.

Нагадаємо, що напівгрупа $S$ із заданою на ній топологією $\tau$ називається:
\begin{itemize}
  \item \emph{ліво-топологічною}, якщо праві внутрішні зсуви  в $(S,\tau)$ неперервні \cite{Ruppert=1984};
  \item \emph{право-топологічною}, якщо ліві внутрішні зсуви  в $(S,\tau)$ неперервні \cite{Ruppert=1984};
  \item \emph{напівтопологічною}, якщо напівгрупова операція на $(S,\tau)$ нарізно неперервна \cite{Ruppert=1984};
  \item \emph{топологічною}, якщо напівгрупова операція на $(S,\tau)$  неперервна \cite{Carruth-Hildebrant-Koch=1983}.
\end{itemize}
Топологія $\tau$ на напівгрупі $S$ називається:
\begin{itemize}
  \item \emph{ліво-неперервною}, якщо $(S,\tau)$ --- ліво-топологічна напівгрупа;
  \item \emph{право-неперервною}, якщо $(S,\tau)$ --- право-топологічна напівгрупа;
  \item \emph{трансляційно-неперервною}, якщо $(S,\tau)$ --- напівтопологічна напівгрупа;
  \item \emph{напівгруповою}, якщо $(S,\tau)$ --- топологічна напівгрупа.
\end{itemize}

У цій праці  досліджуємо топологізацію напівгрупи $\mathscr{O\!\!I}\!_n(L)$. Зокрема, доведено, що кожна $T_1$ ліво-топологічна (право-топологічна) напівгрупа $\mathscr{O\!\!I}\!_n(L)$ є цілком гаусдорфовим, урисонівським, цілком відокремлюваним, розрідженим простором. Доведено, що на напівгрупі $\mathscr{O\!\!I}\!_n(L)$ існує єдина гаусдорфова зліченно компактна (псевдокомпактна) трансляційно-неперервна топологія, яка є компактною.

\section{Топологізація напівгрупи $\mathscr{O\!\!I}\!_n(L)$}\label{section-2}

Нехай $\alpha$ --- довільний елемент напівгрупи $\mathscr{O\!\!I}\!_n(L)$. Позначимо
\begin{equation*}
  {\uparrow}_{\preccurlyeq_l}\alpha=\left\{\beta\in\mathscr{O\!\!I}\!_n(L)\colon \alpha\alpha^{-1}\beta=\alpha\right\} \qquad \hbox{i} \qquad
  {\uparrow}_{\preccurlyeq_r}\alpha=\left\{\beta\in\mathscr{O\!\!I}\!_n(L)\colon \alpha=\beta\alpha^{-1}\alpha\right\}.
\end{equation*}

З леми 1.4.6 \cite{Lawson-1998} випливає

\begin{lemma}\label{lemma-2.1}
Для довільного елемента $\alpha$ напівгрупи $\mathscr{O\!\!I}\!_n(L)$ виконується рівність
\begin{equation*}
  {\uparrow}_{\preccurlyeq}\alpha={\uparrow}_{\preccurlyeq_l}\alpha={\uparrow}_{\preccurlyeq_r}\alpha.
\end{equation*}
\end{lemma}

\begin{lemma}\label{lemma-2.2}
Для довільного елемента $\alpha$ напівгрупи $\mathscr{O\!\!I}\!_n(L)$ множина ${\downarrow}_{\preccurlyeq}\alpha$ скінченна.
\end{lemma}

\begin{proof}
Зафіксуємо довільний елемент $\alpha$ напівгрупи $\mathscr{O\!\!I}\!_n(L)$. Позаяк $\alpha$~--- скінченний частковий порядковий ізоморфізм лінійно впорядкованої множини $L$, то ототожнивши часткову бієкцію $\alpha$ з її графіком $\operatorname{graph}\alpha$ отримуємо, що $\alpha$ --- скінченна підмножина декартового добутку $L\times L$. За твердженням 1(2) з \cite{Gutik-Shchypel=2024}, $\beta\in {\downarrow}_{\preccurlyeq}\alpha$ тоді і лише тоді, коли $\beta\subseteq \alpha$, а отже, множина ${\downarrow}_{\preccurlyeq}\alpha$ скінченна.
\end{proof}

\begin{proposition}\label{proposition-2.3}
Нехай $\tau$~--- ліво-неперервна (право-неперервна) $T_1$-топологія на напівгрупі $\mathscr{O\!\!I}\!_n(L)$. Тоді для довільного елемента $\alpha$ напівгрупи $\mathscr{O\!\!I}\!_n(L)$ існує його відкритий окіл $U(\alpha)$ в $(\mathscr{O\!\!I}\!_n(L),\tau)$ такий, що $U(\alpha)\subseteq {\uparrow}_{\preccurlyeq}\alpha$.
\end{proposition}

\begin{proof}
Зафіксуємо довільний відкритий окіл $W(\alpha)$ елемента $\alpha$ в топологічному просторі $(\mathscr{O\!\!I}\!_n(L),\tau)$. Позаяк $\tau$~--- $T_1$-топологія на $\mathscr{O\!\!I}\!_n(L)$, то усі одноточкові підмножини в $(\mathscr{O\!\!I}\!_n(L),\tau)$ замкнені, а отже, за лемою~\ref{lemma-2.2}, не зменшуючи загальності, можемо вважати, що $W(\alpha)\cap {\downarrow}_{\preccurlyeq}^\circ\alpha=\varnothing$. З неперервності правих зсувів у $(\mathscr{O\!\!I}\!_n(L),\tau)$ випливає, що існує відкритий окіл $U(\alpha)$ елемента $\alpha$ в  $(\mathscr{O\!\!I}\!_n(L),\tau)$ такий, що $U(\alpha)\cdot \alpha^{-1}\alpha\subseteq W(\alpha)$. Якщо $U(\alpha)\nsubseteq {\uparrow}_{\preccurlyeq}\alpha$, то з властивостей напівгрупової операції в $\mathscr{O\!\!I}\!_n(L)$ та леми \ref{lemma-2.1} випливає, що існує елемент $\chi\in U(\alpha)$ такий, що $\chi\cdot \alpha^{-1}\alpha\neq\alpha$, тобто $\chi\cdot \alpha^{-1}\alpha\in {\downarrow}_{\preccurlyeq}^\circ\alpha$, а це суперечить вибору окола $W(\alpha)$. З отриманого протиріччя випливає твердження.

У випадку неперервності лівих зсувів доведення аналогічне.
\end{proof}

\begin{lemma}\label{lemma-2.4}
Нехай $\tau$~--- ліво-неперервна (право-неперервна) $T_1$-топологія на напівгрупі $\mathscr{O\!\!I}\!_n(L)$. Тоді ${\uparrow}_{\preccurlyeq}\alpha$ --- відкрито-замкнена підмножина в $(\mathscr{O\!\!I}\!_n(L),\tau)$ для довільного елемента $\alpha$ напівгрупи $\mathscr{O\!\!I}\!_n(L)$.
\end{lemma}

\begin{proof}
З леми \ref{lemma-2.2} випливає, що ${\uparrow}_{\preccurlyeq}\alpha$ --- відкрита підмножина в $(\mathscr{O\!\!I}\!_n(L),\tau)$ для довільного елемента $\alpha$ напівгрупи $\mathscr{O\!\!I}\!_n(L)$. Позаяк $\tau$~--- $T_1$-топологія на $\mathscr{O\!\!I}\!_n(L)$, то кожна точка $\alpha$ у тополо\-гіч\-ному просторі $\mathscr{O\!\!I}\!_n(L)$ є замкненою множиною. Тоді ${\uparrow}_{\preccurlyeq}\alpha$ --- замкнена підмножина в $(\mathscr{O\!\!I}\!_n(L),\tau)$ як повний прообраз точки $\alpha$ стосовно неперервного правого (лівого) зсуву $\chi\mapsto \chi\cdot\alpha^{-1}\alpha$ ($\chi\mapsto \alpha\alpha^{-1}\cdot\chi$) в $(\mathscr{O\!\!I}\!_n(L),\tau)$.
\end{proof}

Нагадаємо, що топологічний простір $X$ називається:
\begin{itemize}
  \item \emph{урисонівським}, якщо для довільної пари різних точок $x_1,x_2\in X$ існує неперервне відображення $f\colon[0,1]\to X$ таке, що $(x_1)f=0$ i $(x_2)f=1$ \cite{Steen-Seebach=1995};
  \item \emph{цілком гаусдорфовим}, якщо для довільної пари різних точок $x_1,x_2\in X$ існують їхні відкриті околи $U(x_1)$ i $U(x_2)$ такі, що $\operatorname{cl}_X(U(x_1))\cap \operatorname{cl}_X(U(x_2))=\varnothing$ \cite{Urysohn=1925};
  \item \emph{цілком відокремлюваним}, якщо для довільної пари різних точок $x_1,x_2\in X$ існують їхні диз'юнктні відкриті околи $U(x_1)$ i $U(x_2)$ такі, що $U(x_1)\cup U(x_2)=X$ \cite{Steen-Seebach=1995};
  \item \emph{цілком незв'язним}, якщо кожна компонента зв'язності в $X$ є одноелеметною \cite{Engelking=1989};
  \item \emph{розрідженим}, якщо кожна непорожна підмножина $A$ в $X$ містить ізольовану точку в $A$ \cite{Engelking=1989}.
\end{itemize}

Зауважимо, що кожний цілком відокремлюваний простір є цілком незв'язним, урисонівським і цілком гаусдорфовим \cite{Steen-Seebach=1995}.

З леми~\ref{lemma-2.4} випливає

\begin{proposition}\label{proposition-2.5}
Нехай $\tau$~--- ліво-неперервна (право-неперервна) $T_1$-топологія на напівгрупі $\mathscr{O\!\!I}\!_n(L)$. Тоді $(\mathscr{O\!\!I}\!_n(L),\tau)$ --- цілком відокремлюваний простір.
\end{proposition}

\begin{proposition}\label{proposition-2.6}
Якщо $\tau$~--- ліво-неперервна (право-неперервна) $T_1$-топологія на напівгрупі $\mathscr{O\!\!I}\!_n(L)$, то $(\mathscr{O\!\!I}\!_n(L),\tau)$ --- розріджений простір.
\end{proposition}

\begin{proof}
Нехай $A$~--- довільна непорожна підмножина в $(\mathscr{O\!\!I}\!_n(L),\tau)$. Позаяк кожен ранг кожного елемента напівгрупи $\mathscr{O\!\!I}\!_n(L)$ обмежений натуральним числом, то існує елемент $\alpha$ в $\mathscr{O\!\!I}\!_n(L)$ такий, що $A\cap {\uparrow}_{\preccurlyeq}\alpha=\{\alpha\}$. Далі скористаємося твердженням \ref{proposition-2.3}.
\end{proof}

\begin{lemma}\label{lemma-2.7}
Нехай $L$~--- лінійно впорядкована множина. Тоді ${\uparrow}_{\preccurlyeq}\alpha\subseteq \mathscr{I}^n(L)\setminus \mathscr{O\!\!I}\!_n(L)$ в $\mathscr{I}^n(L)$ для довільного $\alpha\in\mathscr{I}^n(L)\setminus \mathscr{O\!\!I}\!_n(L)$.
\end{lemma}

\begin{proof}
Зафіксуємо довільний елемент $\alpha\in\mathscr{I}^n(L)\setminus \mathscr{O\!\!I}\!_n(L)$. З означень напівгруп $\mathscr{I}^n(L)$ i $\mathscr{O\!\!I}\!_n(L)$ випливає, що $\operatorname{rank}\alpha\geq 2$. Тоді існують такі $x_1,x_2\in \operatorname{dom}\alpha$, що $x_1\leqslant x_2$ i $(x_2)\alpha\leqslant (x_1)\alpha$. З означення природного часткового порядку на $\mathscr{I}^n(L)$ (див. \cite{Gutik-Reiter=2009, Gutik-Reiter=2010}) випливає, що для довільного $\beta\in{\uparrow}_{\preccurlyeq}\alpha$ виконуються умови $x_1,x_2\in \operatorname{dom}\beta$ і $(x_2)\alpha,(x_1)\alpha\in \operatorname{ran}\beta$. Позаяк $(x_2)\alpha\leqslant (x_1)\alpha$, то $\beta\notin \mathscr{O\!\!I}\!_n(L)$, звідки випливає твердження леми.
\end{proof}

Позаяк твердження лем~\ref{lemma-2.1}, \ref{lemma-2.2} i \ref{lemma-2.4} виконуються і для напівгрупи $\mathscr{I}^n(L)$ (див. \cite{Gutik-Reiter=2010}), то з леми \ref{lemma-2.7} випливає

\begin{proposition}\label{proposition-2.8}
Нехай $\tau$~--- ліво-неперервна (право-неперервна) $T_1$-топологія на напівгрупі $\mathscr{I}^n(L)$. Тоді $\mathscr{O\!\!I}\!_n(L)$ --- замкнена піднапівгрупа в $(\mathscr{I}^n(L),\tau)$.
\end{proposition}

Надалі під \emph{топологізовною напівгрупою} будемо розуміти напівгрупу із заданою на ній топологією.

\begin{definition}[\cite{Gutik-Pavlyk=2003, Stepp=1975}]\label{definition-2.9}
Нехай $\mathfrak{S}$~--- клас топологізовних напівгруп. Будемо говорити, що (алгеб\-рич\-на) напівгрупа $S$ \emph{алгебрично h-замкнена в класі} $\mathfrak{S}$, якщо $S$ з дискретною топологією є елементом класу $\mathfrak{S}$ і для довільних топологізовної напівгрупи $T\in\mathfrak{S}$ і гомоморфізму $\mathfrak{h}\colon S\to T$ образ $(S)\mathfrak{h}$ є замкненою піднапівгрупою в $T$.
\end{definition}

Нагадаємо \cite{Gutik-Lawson-Repov=2009}, що нескінченна підмножина $D$ напівгрупи $S$ називається \emph{$\omega$-нестій\-кою}, якщо $sB\cup Bs\nsubseteq D$ для довільних $s\in D$ і нескінченої підмножини $B$ у $D$. Будемо говорити, що \emph{напівгрупа $S$ має щільний ряд ідеалів} $J_0\subseteq J_1\subseteq\cdots\subseteq J_m$, $m\in\mathbb{N}$, якщо $J_0$~--- скінченний ідеал в $S$ i $J_k\setminus J_{k-1}$~--- $\omega$-нестійка підмножина в $S$ для довільного $k=1,\ldots,m$.

\begin{theorem}\label{theorem-2.10}
Нехай $(L,\leqslant)$~--- нескінченна лінійно впорядкована множина, $n$ --- довільне натуральне число. Тоді напівгрупа $\mathscr{O\!\!I}\!_n(L)$ алгебрично h-замкнена в класі гаусдорфових напівтопологічних інверсних напівгруп з неперервною інверсією.
\end{theorem}

\begin{proof}
Нехай $T$~--- довільна гаусдорфова напівтопологічна інверсна напівгрупа з неперервною інверсією і $\mathfrak{h}\colon \mathscr{O\!\!I}\!_n(L)\to T$~--- гомоморфізм. Якщо $\mathfrak{h}$~--- анулюючий гомоморфізм, то доводити нічого, оскільки кожна точка в гаусдорфому топологічному просторі $T$ є замкненою множиною. Тому надалі вважатимемо, що $\mathfrak{h}\colon \mathscr{O\!\!I}\!_n(L)\to T$~--- неанулюючий гомоморфізм. За теоремою 3 з \cite{Gutik-Shchypel=2024} гомоморфний образ $(\mathscr{O\!\!I}\!_n(L))\mathfrak{h}$ є піднапівгрупою в $T$, яка має щільний ряд ідеалів, а отже, за тверджен\-ням~10 з \cite{Gutik-Lawson-Repov=2009} $(\mathscr{O\!\!I}\!_n(L))\mathfrak{h}$~--- замкнена піднапівгрупа в $T$.
\end{proof}

\begin{proposition}\label{proposition-2.11}
Нехай $\tau$~--- ліво-неперервна (право-неперервна) $T_1$-топологія на напівгрупі $\mathscr{O\!\!I}\!_n(L)$. Тоді природний частковий порядок $\preccurlyeq$ є замкненим відношенням  на $\mathscr{O\!\!I}\!_n(L)$.
\end{proposition}

\begin{proof}
Зафіксуємо два довільні різні елементи $\alpha$ i $\beta$ напівгрупи $\mathscr{O\!\!I}\!_n(L)$, які непорівняльні стосовно її природного часткового порядку $\preccurlyeq$. Тоді $\alpha\notin{\uparrow}_{\preccurlyeq}\beta$ i $\beta\notin {\uparrow}_{\preccurlyeq}\alpha$. Приймемо $U(\alpha)={\uparrow}_{\preccurlyeq}\alpha\setminus {\uparrow}_{\preccurlyeq}\beta$ i $U(\beta)={\uparrow}_{\preccurlyeq}\beta\setminus {\uparrow}_{\preccurlyeq}\alpha$. З леми~\ref{lemma-2.4} випливає, що $U(\alpha)$ i $U(\beta)$~--- відкрито-замкнені околи точок $\alpha$ i $\beta$ в топологічному просторі $(\mathscr{O\!\!I}\!_n(L),\tau)$, відповідно. За побудовою множини $U(\alpha)$ i $U(\beta)$ маємо, що $U(\alpha)\cap U(\beta)=\varnothing$ i для довільних $\gamma\in U(\alpha)$ i $\delta\in U(\beta)$ елементи $\gamma$ i $\delta$ непорівняльні стосовно природного часткового порядку $\preccurlyeq$ на напівгрупі $\mathscr{O\!\!I}\!_n(L)$. Отже, $(U(\alpha)\times U(\beta))\cap \preccurlyeq=\varnothing$, звідки випливає, що природний частковий порядок $\preccurlyeq$ є замкненим відношенням на  $\mathscr{O\!\!I}\!_n(L)\times \mathscr{O\!\!I}\!_n(L)$ з топологією добутку.
\end{proof}

\section{Про компактні топології на напівгрупі $\mathscr{O\!\!I}\!_n(L)$}

Нагадаємо \cite{Engelking=1989}, що топологічний простір $X$ називається:
\begin{itemize}
  \item \emph{компактним}, якщо довільне відкрите покриття простору $X$ містить скінченне підпокриття;
  \item \emph{зліченно компактним}, якщо довільне зліченне відкрите покриття простору $X$ містить скінченне підпокриття;
  \item \emph{секвенціально компактним}, якщо $X$ --- гаусдорфовий і кожна послідовність точок в $X$ містить збіжну підпослідовність;
  \item \emph{псевдокомпактним}, якщо $X$ --- цілком регулярний і кожна неперервна дійсно визначена функція на $X$ є обмеженою.
\end{itemize}

За теоремою 5 з \cite{Gutik-Pavlyk-Reiter=2009} нескінченна напівгрупа матричних одиниць не занурюється ізоморфно в довільну гаусдорфову зліченно компактну топологічну напівгрупу. Звідси випливає, що аналогічне твердження справджується для напівгрупи $\mathscr{O\!\!I}\!_n(L)$ для довільної нескінченної лінійно впорядкованої множини $(L,\leqslant)$ і довільного натурального числа  $n$.

За твердженням 10 з \cite{Gutik-Reiter=2010} топологія $\tau_\mathrm{c}$ на напівгрупі $\mathscr{I}^n(L)$, породжена базою
\begin{equation*}\label{eq-3.1}
  \mathscr{B}=\left\{U_\alpha(\alpha_1,\ldots,\alpha_k)={\uparrow}_{\preccurlyeq}\alpha\setminus({\uparrow}_{\preccurlyeq}\alpha_1\cup\cdots\cup {\uparrow}_{\preccurlyeq}\alpha_k)\colon \alpha_i\in{\uparrow}_{\preccurlyeq}\alpha,\, \alpha,\alpha_i\in \mathscr{I}^n(L), i=1,\ldots,k\right\},
\end{equation*}
гаусдорфова, компактна, трансляційно-неперервна і, крім того, інверсія стосовно неї також є неперервною. Також через $\tau_\mathrm{c}$ позначимо топологію на напівгрупі $\mathscr{O\!\!I}\!_n(L)$, індуковану з топологічного простору $(\mathscr{I}^n(L),\tau_\mathrm{c})$. Очевидно, що топологія $\tau_\mathrm{c}$ на напівгрупі $\mathscr{O\!\!I}\!_n(L)$ породжується базою
\begin{equation*}\label{eq-3.2}
  \mathscr{B}_\mathrm{c}=\left\{U_\alpha(\alpha_1,\ldots,\alpha_k)={\uparrow}_{\preccurlyeq}\alpha\setminus({\uparrow}_{\preccurlyeq}\alpha_1\cup\cdots\cup {\uparrow}_{\preccurlyeq}\alpha_k)\colon \alpha_i\in{\uparrow}_{\preccurlyeq}\alpha,\, \alpha,\alpha_i\in \mathscr{O\!\!I}\!_n(L), i=1,\ldots,k\right\}.
\end{equation*}

З твердженням 10 \cite{Gutik-Reiter=2010} і твердження \ref{proposition-2.8} випливає

\begin{proposition}\label{proposition-3.1}
Для довільного натурального числа $n$, $(\mathscr{O\!\!I}\!_n(L),\tau_\mathrm{c})$ --- гаусдорфова компактна напівтопологічна напівгрупа з неперервною інверсією.
\end{proposition}

\begin{remark}\label{remark-3.2}
З леми \ref{lemma-2.4} та означення топології $\tau_\mathrm{c}$ на напівгрупі $\mathscr{O\!\!I}\!_n(L)$ випливає, що для довільної ліво-неперервної (право-неперервної) $T_1$-топології $\tau$ на $\mathscr{O\!\!I}\!_n(L)$ виконується включення $\tau_\mathrm{c}\subseteq\tau$.
\end{remark}

\begin{theorem}\label{theorem-3.3}
Нехай $(L,\leqslant)$~--- нескінченна лінійно впорядкована множина, $n$ --- довільне натуральне число і $\tau$~--- гаусдофрова топологія на напівгрупі $\mathscr{O\!\!I}\!_n(L)$. Тоді такі умови еквівалентні:
\begin{enumerate}
  \item[$(1)$] $(\mathscr{O\!\!I}\!_n(L),\tau)$~--- компактна напівтопологічна напівгрупа;
  \item[$(2)$] $(\mathscr{O\!\!I}\!_n(L),\tau)$ топологічно ізоморфна $(\mathscr{O\!\!I}\!_n(L),\tau_{\mathrm{c}})$;
  \item[$(3)$] $(\mathscr{O\!\!I}\!_n(L),\tau)$~--- компактна напівтопологічна напівгрупа з неперервною інверсією;
  \item[$(4)$] $(\mathscr{O\!\!I}\!_n(L),\tau)$~--- зліченно компактна напівтопологічна напівгрупа;
  \item[$(5)$] $(\mathscr{O\!\!I}\!_n(L),\tau)$~--- секвенціально компактна напівтопологічна напівгрупа;
  \item[$(6)$] $(\mathscr{O\!\!I}\!_n(L),\tau)$~--- псевдокомпактна напівтопологічна напівгрупа;
  \item[$(7)$] $(\mathscr{O\!\!I}\!_n(L),\tau)$~--- компактна ліво-топологічна напівгрупа;
  \item[$(8)$] $(\mathscr{O\!\!I}\!_n(L),\tau)$~--- компактна право-топологічна напівгрупа;
  \item[$(9)$] $(\mathscr{O\!\!I}\!_n(L),\tau)$~--- зліченно компактна ліво-топологічна напівгрупа;
  \item[$(10)$] $(\mathscr{O\!\!I}\!_n(L),\tau)$~--- зліченно компактна право-топологічна напівгрупа;
  \item[$(11)$] $(\mathscr{O\!\!I}\!_n(L),\tau)$~--- секвенціально компактна ліво-топологічна напівгрупа;
  \item[$(12)$] $(\mathscr{O\!\!I}\!_n(L),\tau)$~--- секвенціально компактна право-топологічна напівгрупа;
  \item[$(13)$] $(\mathscr{O\!\!I}\!_n(L),\tau)$~--- псевдокомпактна ліво-топологічна напівгрупа;
  \item[$(14)$] $(\mathscr{O\!\!I}\!_n(L),\tau)$~--- псевдокомпактна право-топологічна напівгрупа.
\end{enumerate}
\end{theorem}

\begin{proof}
Імплікації  $(1)\Rightarrow(7)$, $(1)\Rightarrow(8)$, $(2)\Rightarrow(1)$, $(2)\Rightarrow(3)$, $(3)\Rightarrow(1)$, $(4)\Rightarrow(9)$, $(4)\Rightarrow(10)$, $(5)\Rightarrow(11)$. $(5)\Rightarrow(12)$, $(6)\Rightarrow(13)$, $(6)\Rightarrow(14)$, $(11)\Rightarrow(9)$, $(12)\Rightarrow(10)$ очевидні. Імплікації $(1)\Rightarrow(4)$, $(2)\Rightarrow(4)$, $(1)\Rightarrow(6)$ і $(5)\Rightarrow(4)$ випливають з властивостей топологічних просторів (див. \cite[глава 3]{Engelking=1989}).

$(7)\Rightarrow(2)$
Із зауваження \ref{remark-3.2} випливає, що $\tau_\mathrm{c}\subseteq\tau$, оскільки $\tau$ --- гаусдорфова компактна ліво-неперервна топологія на напівгрупі $\mathscr{O\!\!I}\!_n(L)$. Тоді скориставшись наслідком 3.1.14 з \cite{Engelking=1989}, отримуємо, що $\tau=\tau_\mathrm{c}$.
Доведення імплікації $(8)\Rightarrow(2)$ аналогічне.

$(9)\Rightarrow(7)$
Нехай $\tau$~--- гаусдорфова зліченно компактна діво-неперервна топологія на напівгрупі $\mathscr{O\!\!I}\!_n(L)$. Доведення проведемо індукцією по $n$. Очевидно, що напівгрупа $\mathscr{O\!\!I}\!_1(L)$ ізо\-морфна напівгрупі $L\times L$-матричних одиниць. За лемою~\ref{lemma-2.4} ${\uparrow}_{\preccurlyeq}\alpha=\{\alpha\}$ --- відкрито-замкнена підмножина в $(\mathscr{O\!\!I}\!_1(L),\tau)$ для довільного ненульового елемента $\alpha$ напівгрупи $\mathscr{O\!\!I}\!_1(L)$, а отже, з теореми 3.10.3 \cite{Engelking=1989} випливає, що кожен відкритий окіл нуля в $(\mathscr{O\!\!I}\!_1(L),\tau)$ має скінченне доповнення. Тоді за теоремою 2 \cite{Gutik-Pavlyk=2005} ліво-топологічна напівгрупа $(\mathscr{O\!\!I}\!_1(L),\tau)$ є компактною напівтопологічною напівгрупою.

Далі доведемо, що з того, що для довільного натурального числа $k<n$ зліченно компактна ліво-топологічна напівгрупа $(\mathscr{O\!\!I}\!_k(L),\tau)$ є компактною, випливає, що зліченно компактна напівгрупа $(\mathscr{O\!\!I}\!_n(L),\tau)$ також є компактною. З леми \ref{lemma-2.4} випливає, що $\mathscr{O\!\!I}\!_{n-1}(L)$~--- замкнена піднапівгрупа в $(\mathscr{O\!\!I}\!_n(L),\tau)$, а тоді за теоремою 3.10.4 \cite{Engelking=1989} маємо, що $\mathscr{O\!\!I}\!_{n-1}(L)$~--- зліченно компактний підпростір у $(\mathscr{O\!\!I}\!_n(L),\tau)$. Скориставшись припущенням індукції, отримуємо, що піднапівгрупа $\mathscr{O\!\!I}\!_{n-1}(L)$ в $(\mathscr{O\!\!I}\!_n(L),\tau)$ компактна. Припустимо, що простір $(\mathscr{O\!\!I}\!_n(L),\tau)$ не є компактним. Тоді існує відкрите покриття $\mathscr{U}$ простору $(\mathscr{O\!\!I}\!_n(L),\tau)$, що не містить скінченне підпокриття. З компактності підпростору $\mathscr{O\!\!I}\!_{n-1}(L)$ в $(\mathscr{O\!\!I}\!_n(L),\tau)$ випливає існування скінченної підсім'ї $\mathscr{U}_0$ в $\mathscr{U}$, яка є покриттям простору $\mathscr{O\!\!I}\!_{n-1}(L)$. Тоді $\bigcup\mathscr{U}_0\supseteq \mathscr{O\!\!I}\!_{n-1}(L)$, однак за припущенням множина $\mathscr{O\!\!I}\!_{n}(L)\setminus \bigcup\mathscr{U}_0$ нескінченна та замкнена в $(\mathscr{O\!\!I}\!_n(L),\tau)$. За лемою \ref{lemma-2.4} усі точки множини $\mathscr{O\!\!I}\!_n(L)\setminus\mathscr{O\!\!I}\!_{n-1}(L)$, а отже, і множини $\mathscr{O\!\!I}\!_{n}(L)\setminus \bigcup\mathscr{U}_0$, є ізольованими в $(\mathscr{O\!\!I}\!_n(L),\tau)$. Отже, топологічний простір $(\mathscr{O\!\!I}\!_n(L),\tau)$ містить нескінченну зліченну замкнену множину ізольованих точок, а це суперечить тому, що він є зліченно компактним (див. теорему 3.10.3 в \cite{Engelking=1989}). З отриманого протиріччя випливає наша імплікація. Доведення імплікацій $(4)\Rightarrow(1)$ і $(10)\Rightarrow(8)$ аналогічні.

$(1)\Rightarrow(5)$
За твердженням \ref{proposition-2.6} топологічний простір $(\mathscr{O\!\!I}\!_n(L),\tau)$ є розрідженим.
Позаяк кожен зліченно компактний розріджений $T_3$-простір, а отже, і гаусдорфовий компактний розріджений простір, є секвенціально компактним (див. \cite[теорема 5.7]{Vaughan=1984}), то $(\mathscr{O\!\!I}\!_n(L),\tau)$ є секвенціально компактним прос\-тором. Доведення імплікацій $(7)\Rightarrow(11)$ і $(8)\Rightarrow(12)$ аналогічні.

$(13)\Rightarrow(7)$
Нехай $\tau$~--- псевдокомпактна ліво-неперервна топологія на напівгрупі $\mathscr{O\!\!I}\!_n(L)$.
Провівши аналогічні міркування доведення індукцією по $n$, як і в доведенні імплікації $(9)\Rightarrow(7)$, отримуємо, що  топологіч\-ний простір $(\mathscr{O\!\!I}\!_n(L),\tau)$ містить нескінченну зліченну замкнену множину ізольованих точок $\left\{\alpha_i\colon i\in\omega\right\}=\mathscr{O\!\!I}\!_n(L)\setminus\bigcup\mathscr{U}_0$. Означимо відображення $\mathfrak{f}\colon (\mathscr{O\!\!I}\!_n(L),\tau)\to \mathbb{R}$ у множину дійсних чисел $\mathbb{R}$ зі звичайною топологією за формулою
\begin{equation*}
  (\alpha)\mathfrak{f}=
  \left\{
    \begin{array}{cl}
      -1, & \hbox{якщо~} \alpha\in\bigcup\mathscr{U}_0;\\
      i,  & \hbox{якщо~} \alpha=\alpha_i \hbox{~для деякого~} i\in\omega.
    \end{array}
  \right.
\end{equation*}
Очевидно, що так означене відображення $\mathfrak{f}$ неперервне та необмежене, а це суперечить псевдокомпактності простору $(\mathscr{O\!\!I}\!_n(L),\tau)$. З отриманого протиріччя випливає наша імплікація. Доведення імплікацій $(6)\Rightarrow(1)$ і $(14)\Rightarrow(8)$ аналогічні.
\end{proof}

З теореми \ref{theorem-3.3} та теореми 1 \cite{Gutik-Reiter=2010} випливає такий наслідок.

\begin{corollary}\label{corollary-3.4}
Нехай $(L,\leqslant)$~--- нескінченна лінійно впорядкована множина, $n$ --- довільне натуральне число.
На напівгрупі $\mathscr{O\!\!I}\!_n(L)$ існує єдина гаусдорфова компактна трансляційно-неперервна топологія, яка єдиним чином продожується до гаусдорфової компактної трансляційно-неперервної топології на напівгрупі $\mathscr{I}^n(L)$.
\end{corollary}

\begin{remark}
Зауважимо, що еквівалентність тверджень $(1)-(14)$ теореми~\ref{theorem-3.3} для напівгрупи $\mathscr{I}^n(L)$ доводиться аналогічно.
\end{remark}

Нагадаємо \cite{Engelking=1989}, що топологічний простір $X$ називається:
\begin{itemize}
  \item \emph{локально компактним}, якщо кожна його точка $x$ має відкритий окіл $U(x)$ з компактним замиканням $\overline{U(x)}$;
  \item \emph{нуль-вимірним}, якщо в $X$ існує база, що складається з відкрито-замкнених підмножин.
\end{itemize}

\begin{proposition}\label{proposition-3.5}
Нехай $\tau$~--- ліво-неперервна (право-неперервна) локально компактна $T_1$-топологія на напівгрупі $\mathscr{O\!\!I}\!_n(L)$. Тоді $(\mathscr{O\!\!I}\!_n(L),\tau)$ --- нуль-вимірний простір.
\end{proposition}

\begin{proof}
З твердження \ref{proposition-2.5} випливає, що топологічний простір $(\mathscr{O\!\!I}\!_n(L),\tau)$ є цілком гаусдорфовим, а отже, за теоремою 3.3.1 з \cite{Engelking=1989} $(\mathscr{O\!\!I}\!_n(L),\tau)$ є цілком регулярним простором.

Зафіксуємо довільну точку $\alpha\in \mathscr{O\!\!I}\!_n(L)$. Нехай $U(\alpha)$~--- відкритий окіл точки $\alpha$ в просторі $(\mathscr{O\!\!I}\!_n(L),\tau)$ з компактним замиканням $\overline{U(\alpha)}$. Тоді за лемою~\ref{lemma-2.4} маємо, що  $V(\alpha)={\uparrow}_{\preccurlyeq}\alpha\cap U(\alpha)$~--- відкритий окіл точки $\alpha$ з компактним замиканням $\overline{V(\alpha)}={\uparrow}_{\preccurlyeq}\alpha\cap\overline{U(\alpha)}$. Якщо $\overline{V(\alpha)}\setminus V(\alpha)=\varnothing$, то доводити нічого. Тому будемо вважати, що $\overline{V(\alpha)}\setminus V(\alpha)\neq\varnothing$. За лемою~\ref{lemma-2.4} сім'я
$\mathscr{U}=\big\{{\uparrow}_{\preccurlyeq}\beta\colon \beta\in \overline{V(\alpha)}\setminus V(\alpha)\big\}$ є відкритим покриттям множини $\overline{V(\alpha)}\setminus V(\alpha)$. Позаяк множина $\overline{V(\alpha)}\setminus V(\alpha)$ компактна, то $\mathscr{U}$ містить скінченне підпокриття $\mathscr{U}_0=\left\{{\uparrow}_{\preccurlyeq}\beta_1,\ldots,{\uparrow}_{\preccurlyeq}\beta_k\right\}$.  Тоді за побудовою множина
\begin{equation*}
W(\alpha)=\overline{V(\alpha)}\setminus({\uparrow}_{\preccurlyeq}\beta_1\cup\cdots\cup{\uparrow}_{\preccurlyeq}\beta_k)= V(\alpha)\setminus({\uparrow}_{\preccurlyeq}\beta_1\cup\cdots\cup{\uparrow}_{\preccurlyeq}\beta_k)
\end{equation*}
містить точку $\alpha$, і за лемою~\ref{lemma-2.4} є її відкрито-замкненим околом. Очевидно, що $W(\alpha)\subseteq U(\alpha)$. Звідси випливає, що топологічний простір $(\mathscr{O\!\!I}\!_n(L),\tau)$ має відкрито-замкнену базу, а отже, є нуль-вимірним.
\end{proof}

З наступного прикладу випливає, що для довільного натурального числа $n$ у локально ком\-пакт\-ному випадку з неперервності правих (лівих) зсувів не випливає нарізна неперервність напівгрупової операції на топологізовній напівгрупі $\mathscr{O\!\!I}\!_n(L)$.

\begin{example}
Для довільного елемента $x\in L$ позначимо $\mathscr{L}^x=\left\{\alpha=
\left(
\begin{smallmatrix}
  x \\
  y
\end{smallmatrix}
\right)\colon y\in L\right\}$. Очевидно, що $\mathscr{L}^x\subseteq \mathscr{O\!\!I}\!_1(L)$ для довільного $x\in L$. Зафіксуємо довільний елемент $x_0\in L$. Топологію $\tau^{x_0}$ на $\mathscr{O\!\!I}\!_n(L)$ визначимо так:
\begin{itemize}
  \item[(1)] усі ненульові елементи напівгрупи $\mathscr{O\!\!I}\!_n(L)$ є ізольованими точками в просторі $(\mathscr{O\!\!I}\!_n(L),\tau^{x_0})$;
  \item[(2)] сім'я $\mathscr{B}^{x_0}=\left\{ U^{x_0}_{\alpha_1,\ldots,\alpha_k}=\{\boldsymbol{0}\}\cup\mathscr{L}^{x_0}\setminus\{\alpha_1,\ldots,\alpha_k\} \colon \alpha_1,\ldots,\alpha_k\in\mathscr{L}^{x_0} \right\}$ є базою топології $\tau^{x_0}$ в нулі $\boldsymbol{0}$ напівгрупи $\mathscr{O\!\!I}\!_n(L)$.
\end{itemize}

Очевидно, що $(\mathscr{O\!\!I}\!_n(L),\tau^{x_0})$ --- локально компактний гаусдорфовий топологічний простір. Ліві зсуви в топологізовній напівгрупі $(\mathscr{O\!\!I}\!_n(L),\tau^{x_0})$ не є неперервними, оскільки $\alpha\cdot U^{x_0}_{\alpha_1,\ldots,\alpha_k}\nsubseteq U^{x_0}_{\alpha'_1,\ldots,\alpha'_p}$ для довільних $\alpha_1,\ldots,\alpha_k,\alpha'_1,\ldots,\alpha'_p\in \mathscr{L}^{x_0}$ у випадку, коли $\operatorname{dom}\alpha\neq\operatorname{ran}\alpha=\{x_0\}$. Однак, $U^{x_0}_\beta\cdot\alpha=\{\boldsymbol{0}\}$, де $\beta=
\left(
\begin{smallmatrix}
  x_0 \\
  y
\end{smallmatrix}
\right)$ i $y\notin\operatorname{dom}\alpha$, звідки випливає, що праві зсуви в $(\mathscr{O\!\!I}\!_n(L),\tau^{x_0})$ є неперервними, оскільки всі ненульові елементи напівгрупи $\mathscr{O\!\!I}\!_n(L)$ є ізольованими точками в просторі $(\mathscr{O\!\!I}\!_n(L),\tau^{x_0})$.
\end{example}

Нехай $X$~--- непорожня множина та $\mathcal{O}\notin X\times X$. На $\mathfrak{B}_X=X\times X\sqcup\{\mathcal{O}\}$ означимо напівгрупову операцію так:
\begin{equation*}
  (x,y)\cdot(u,v)=
  \left\{
    \begin{array}{cl}
      (x,v),       & \hbox{якщо~} y=u;\\
      \mathcal{O}, & \hbox{якщо~} y\neq u
    \end{array}
  \right.
 \qquad \hbox{i} \qquad (x,y)\cdot\mathcal{O}=\mathcal{O}\cdot (x,v)=\mathcal{O}\cdot\mathcal{O}=\mathcal{O},
\end{equation*}
де $x,y,u,v\in X$. Множина $\mathfrak{B}_X$ з вище означеною напівгруповою оперцією називається \emph{напівгрупою $X{\times} X$-матричних одиниць} \cite{Clifford-Preston-1961}, а у випадку, коли множина $X$ нескінченна, то будемо говорити, що $\mathfrak{B}_X$ --- \emph{нескінченна напівгрупа матричних одиниць}. Очевидно, що напівгрупи $\mathscr{I}^1(L)$ і $\mathscr{O\!\!I}\!_1(L)$ ізоморфні напівгрупі $L{\times} L$-матричних одиниць.

Нагадаємо \cite{Clifford-Preston-1961}, що гомоморфізм з напівгрупи $S$ у напівгрупу $T$ називається \emph{анулюючим}, якщо існує такий елемент $t_0\in T$, що $(s)\mathfrak{h}=t_0$ для всіх $s\in S$.

\begin{lemma}\label{lemma-3.6}
Нехай $(L,\leqslant)$~--- нескінченна лінійно впорядкована множина, $n$ --- довільне натуральне число i $S$~--- напівгрупа. Якщо $\mathfrak{h}\colon \mathscr{O\!\!I}\!_n(L)\to S$~--- неанулюючий гомоморфізм, то образ $(\mathscr{O\!\!I}\!_n(L))\mathfrak{h}$ містить напівгрупу, яка ізоморфна нескінченній  напівгрупі матричних одиниць.
\end{lemma}

\begin{proof}
Якщо гомоморфізм $\mathfrak{h}\colon \mathscr{O\!\!I}\!_n(L)\to S$ ін'єктивний, то образ $(\mathscr{O\!\!I}\!_1(L))\mathfrak{h}$  ізоморфний напівгрупі $\mathscr{O\!\!I}\!_1(L)$, а отже, $(\mathscr{O\!\!I}\!_n(L))\mathfrak{h}$ містить напівгрупу, яка ізоморфна нескінченній  напівгрупі матричних одиниць. Тому надалі будемо вважати, що гомоморфізм $\mathfrak{h}$ не є ін'єктивним. Тоді гомоморфізм $\mathfrak{h}$ породжує на напівгрупі $\mathscr{O\!\!I}\!_n(L)$ конґруенцію $\mathcal{C}_{\mathfrak{h}}$. За теоремою~2 з \cite{Gutik-Shchypel=2024} кожна конґруенція на напівгрупі $\mathscr{O\!\!I}\!_n(L)$ є конґруенцією Ріса, а отже, конґруенція $\mathcal{C}_{\mathfrak{h}}$ породжується ідеалом $I_k=\mathscr{O\!\!I}\!_k(L)$ для деякого натурального числа $k\in\{1,\ldots,n-1\}$. Тоді для довільних $\alpha, \beta\in \mathscr{O\!\!I}\!_{k+1}(L)$ i $\gamma\in\mathscr{O\!\!I}\!_k(L)$ з визна\-чен\-ня напівгрупової операції на $\mathscr{O\!\!I}\!_n(L)$ випливає, що $\alpha \beta\in \mathscr{O\!\!I}\!_{k+1}(L)$ у випадку $\operatorname{ran}\alpha=\operatorname{dom}\beta$, i $\alpha \beta\in \mathscr{O\!\!I}\!_{k}(L)$ у випадку $\operatorname{ran}\alpha\neq\operatorname{dom}\beta$, і $\alpha\gamma, \gamma\alpha \in \mathscr{O\!\!I}\!_{k}(L)$. Отже, гомоморфний  образ $(\mathscr{O\!\!I}\!_{k+1}(L))\mathfrak{h}$ містить напівгрупу ізоморфну нескінченній  напівгрупі матричних одиниць, оскільки образ $(\mathscr{O\!\!I}\!_{k}(L))\mathfrak{h}$ є нулем піднапівгрупи $(\mathscr{O\!\!I}\!_{n}(L))\mathfrak{h}$.
\end{proof}

За теоремою 6 з \cite{Gutik-Pavlyk-Reiter=2009} кожний гомоморфізм з нескінченної напівгрупи матричних одиниць у гаусдорфову зліченно компактну топологічну напівгрупу є анулюючим. Отже, леми~\ref{lemma-3.6} випливає така теорема.

\begin{theorem}\label{theorem-3.7}
Нехай $(L,\leqslant)$~--- нескінченна лінійно впорядкована множина, $n$ --- довільне натуральне число. Тоді кожний гомоморфізм $\mathfrak{h}\colon \mathscr{O\!\!I}\!_n(L)\to S$ у гаусдорфову зліченно компактну топологічну напівгрупу $S$ є анулюючим.
\end{theorem}

Нагадаємо \cite{DeLeeuw-Glicksberg-1961}, що \emph{компактифікацією Бора} топологічної напівгрупи $S$ називається пара $(\mathfrak{b}, \mathbf{B}(S))$ така, що $\mathbf{B}(S)$ --- компактна напівгрупа, $\mathfrak{b}\colon  S\to \mathbf{B}(S)$ --- неперервний гомоморфізм, і якщо $\mathfrak{g}\colon S\to T$~--- неперервний гомоморфізм напівгрупи $S$ у компактну напівгрупу $T$, то існує єдиний неперервний гомоморфізм $\mathfrak{f}\colon \mathbf{B}(S)\to T$ такий, що діаграма
\begin{equation*}
\xymatrix{
 S\ar[r]^{\!\!\mathfrak{b}}\ar[d]_{\mathfrak{g}}  & \mathbf{B}(S)\ar@{-->}[dl]^{f}\\
 T &}
\end{equation*}
комутативна.

З теореми~\ref{theorem-3.7} випливає

\begin{theorem}\label{theorem-3.8}
Нехай $(L,\leqslant)$~--- нескінченна лінійно впорядкована множина, $n$ --- довільне натуральне число. Тоді компактифікація Бора гаусдорфової топологічної напівгрупи $\mathscr{O\!\!I}\!_n(L)$ є тривіальною напівгрупою.
\end{theorem}




\end{document}